\documentclass{scrartcl}

\usepackage{lmodern}
\usepackage[utf8]{inputenc}
\usepackage[T1]{fontenc}
\usepackage[english]{babel}
\usepackage{csquotes}
\usepackage{dsfont}
\usepackage{eucal}

\usepackage[backend=biber,style=alphabetic,eprint=false,url=false,giveninits=true,sorting=anyt,date=year,maxbibnames=6]{biblatex}
\usepackage{amsmath,amssymb,amsthm}
\usepackage{mathtools,thmtools,stmaryrd,esint}
\usepackage{mathrsfs}
\usepackage{bm} 

\usepackage{hyperref}
\usepackage[shortlabels]{enumitem}
\usepackage[nameinlink]{cleveref}
\usepackage[svgnames]{xcolor}

\usepackage{longtable,tabu}
\usepackage{mdframed}

\usepackage{todonotes} 

\numberwithin{equation}{section}
 \hypersetup{
     breaklinks=true,
     colorlinks=true,
     linkcolor=red,
     filecolor=blue,
     citecolor=ForestGreen,      
     urlcolor=magenta,
     }

\newcommand{\N}{\mathbb{N}}

\newcommand{\R}{\mathbb{R}}

\newcommand{\probspace}{\mathscr{P}}
\newcommand{\contspace}{\mathscr{C}}
\newcommand{\ener}{E}
\newcommand{\MOT}{\mathsf{MOT}}
\newcommand{\OT}{\mathsf{OT}}

\newcommand{\mres}{\mathbin{\vrule height 1.6ex depth 0pt width
0.13ex\vrule height 0.13ex depth 0pt width 1.3ex}}
\newcommand{\eps}{\varepsilon}
\newcommand{\hdm}{{\mathscr H}}
\newcommand{\lbm}{{\mathscr L}}

\newcommand{\dd}{\mathop{}\mathopen{}\mathrm{d}}
\newcommand{\entropy}{\mathrm{Ent}}

\newcommand{\mbf}[1]{\bm{#1}}
\newcommand{\loc}{\mathrm{loc}}

\DeclarePairedDelimiter\abs{\lvert}{\rvert}
\DeclarePairedDelimiter\norm{\lVert}{\rVert}

\DeclareMathOperator{\lipconst}{Lip}

\DeclareMathOperator{\diam}{diam}
\DeclareMathOperator{\tr}{Tr}

\DeclareMathOperator{\id}{Id}

\DeclareMathOperator*{\esssup}{ess\,sup}

\DeclareMathOperator{\spt}{spt}

\declaretheorem[name=Theorem,within=section]{theorem}
\declaretheorem[name=Lem\-ma,numberlike=theorem]{lemma}
\declaretheorem[name=Proposition,numberlike=theorem]{proposition}

\declaretheorem[name=Definition,numberlike=theorem,style=definition]{definition}
\declaretheorem[name=Remark,numberlike=theorem,style=remark]{remark}
\declaretheorem[name=Example,numberlike=theorem,style=remark]{example}

\let\svthefootnote\thefootnote\newcommand\freefootnote[1]{\let\thefootnote\relax\footnotetext{#1}\let\thefootnote\svthefootnote}

\allowdisplaybreaks[2]

\begin{document}

\title{Convergence rate of entropy-regularized multi-marginal optimal transport costs}
\author{Luca Nenna\thanks{Universit\'e Paris-Saclay, CNRS, Laboratoire de math\'ematiques d'Orsay, 91405, Orsay, France. 
email: luca.nenna@universite-paris-saclay.fr}\;  and  Paul Pegon\thanks{CEREMADE, Université Paris-Dauphine, Université PSL, CNRS, Mokaplan, Inria Paris, 75016 Paris, France. email: pegon@ceremade.dauphine.fr}}

\date{\today}

\maketitle

\begin{abstract}
    We investigate the convergence rate of multi-marginal optimal transport costs that are regularized with the Boltzmann-Shannon entropy, as the noise parameter $\varepsilon$ tends to $0$. We establish lower and upper bounds on the difference with the unregularized cost of the form $C\varepsilon\log(1/\varepsilon)+O(\varepsilon)$ for some explicit dimensional constants $C$ depending on the marginals and on the ground cost, but not on the optimal transport plans themselves. Upper bounds are obtained for Lipschitz costs or locally semi-concave costs for a finer estimate, and lower bounds for $\contspace^2$ costs satisfying some signature condition on the mixed second derivatives that may include degenerate costs, thus generalizing results previously  in the two marginals case and for non-degenerate costs. We obtain in particular matching bounds in some typical situations where the optimal plan is deterministic.
\end{abstract}

\vskip\baselineskip\noindent


\tableofcontents

\section*{Notations}
In all the article, $N \in \N^*$ denotes the dimension of the ambient space $\R^N$ and $m \in \N$ is an integer such that $m\geq 2$.
\newline\newline
\begin{tabu}  {X[2,c,m]  X[12,l,p]}
$B^d_r(x)$ & open Euclidean ball of radius $r$ centered at $x$ in $\R^d$, dropping the supscript $d$ when $d=N$;\\
$\omega_d$ & $d$-dimensional volume of $B^d_1(0)$;\\
$X_i$ & a subset of $\R^N$ for any index $i\in \{1, \ldots, m\}$\\
$\mbf X$ & product $X_1\times \ldots \times X_m$ whenever the $X_i$'s are $m$ subsets of $\R^N$;\\
$\mbf A_{-i}$ & product $\prod_{1\leq j\leq m, j\neq i} A_j$ if $\mbf A = A_1\times \ldots \times A_m \subseteq \mbf X$ and $i\in\{1,\ldots,m\}$;\\
$x_i,x,\mbf{x}$ & a point in $X_i$, in some $X_j, j\in \{1,\ldots,m\}$, and in $\mbf X$ respectively;\\
$x_q$ & $(x_i)_{i\in q}$ if $q \subseteq \{1, \ldots, m\}$ and $\mbf x \in \mbf X$;\\
$e_i$ & $i$-th coordinate map $e_i : \mbf x = (x_1,\ldots,x_m) \mapsto x_i$;\\
$A \Subset B$ & $A\subseteq K \subseteq B$ for some compact set $K$;\\
$\abs{\cdot}$ & Euclidean norm on $\R^N$;\\
$\norm{\cdot}$ & norm on $\R^N\times \ldots \times \R^N$ defined by $\norm{\mbf x} = \max_{1\leq i \leq m} \abs{x_i}$ if $\mbf x = (x_1,\ldots,x_m)$;\\
$B_r(\mbf x)$ & open ball of radius $r$ centered at $\mbf x\in (\R^N)^m$ for the above norm;\\
$\contspace_\loc^{0,1}(X)$ & space of real-valued locally Lipschitz functions on $X$ which is a sub-manifold of $\R^N$ or $(\R^N)^m$;\\
$[f]_{\contspace^{0,1}(X)}$ & Lipschitz constant of $f : X \to \R$ where $X$ is a subset of $\R^N$ or $(\R^N)^m$ for the above norms;\\
$\contspace_\loc^{1,1}(X)$ & space of differentiable real-valued functions on $X$, a sub-manifold of $\R^N$ or $(\R^N)^m$, with locally Lipschitz differential;\\
$\probspace(X)$ & space of probability measures on a metric space $X$;\\
$\norm{\cdot}_{L^p(\mu)}$ & $L^p$ norm induced by a measure $\mu$, where $p\in[1,+\infty]$;\\
$\spt{\mu}$ & support of the measure $\mu$;\\
$\mu \mres A$ & restriction of the Borel measure $\mu$ to the Borel set $A$ defined by $\mu \mres A(E) = \mu(A\cap E)$ for every $E$;\\
$\hdm^s_X$ & $s$-dimensional Hausdorff measure on the metric space $X$ endowed with the Borel $\sigma$-algebra (the subscript $X$ will often be dropped);\\
$M_N(\R)$ & space of real matrices of size $N \times N$, endowed with the Frobenius norm induced by the scalar product $A \cdot B \coloneqq\tr(A^T B)$, for $A,B\in M_N(\R)$;\\
$S_N(\R)$ & subspace of real symmetric matrices of size $N \times N$;\\
$\Delta_P$ & simplex of P-uples $t = (t_p)_{p\in P}$ such that $\forall p, t_p\geq 0$ and $\sum_{p\in P}t_p=1$.
\end{tabu}

\section{Introduction}

We consider a $m-$uple of probability measures $\mu_{i}$ compactly supported on sub-manifolds $X_i\subseteq\R^N$of dimension $d_i$ and a cost function 
 $c : X_1 \times \ldots \times X_m  \to \R_+$. The Entropic Multi-Marginal Optimal Transport problem is defined as :
\begin{equation}\label{MEOTpb}\tag{MOT$_\eps$}
\MOT_\eps\coloneqq\inf\left\{ \int_{X_1 \times \ldots \times X_m} c\dd\gamma + \eps \entropy(\gamma | \otimes_{i=1}^m \mu_i)\;|\;\gamma \in \Pi(\mu_1, \ldots, \mu_m)\right\},
\end{equation}
where $\Pi(\mu_1, \ldots, \mu_m)$ denotes the set of all probability measures $\gamma$ having $\mu_i$ as $i$-th marginal, i.e.\ $(e_i)_\sharp \gamma = \mu_i$ where $e_i : (x_1,\ldots,x_m) \mapsto  x_i$, for every $i \in \{1,\ldots,m\}$. The classical multi-marginal optimal transport problem corresponds to the case where $\eps=0$.
In the last decade, these two classes of problems (entropic optimal transport and multi-marginal optimal transport) have witnessed a  growing interest and they are now an active research topic.

Entropic optimal transport (EOT)  has found applications and proved to be an efficient way to approximate Optimal Transport (OT) problems, especially from a computational viewpoint. Indeed, when it comes to solving EOT by alternating Kullback-Leibler projections on the two marginal constraints, by the algebraic properties of the entropy such iterative projections correspond to the celebrated Sinkhorn's algorithm \cite{sinkhornRelationshipArbitraryPositive1964}, applied in this framework in the pioneering works \cite{cuturiSinkhornDistancesLightspeed2013, benamouIterativeBregmanProjections2015}. The simplicity and the good convergence guarantees (see \cite{franklinScalingMultidimensionalMatrices1989, marinoOptimalTransportApproach2020,carlierLinearConvergenceMultimarginal2022,ghosalConvergenceRateSinkhorn2022}) of this method compared to the algorithms used for the OT problems, then determined the success of EOT for applications in machine learning, statistics, image processing, language processing and other areas (see the monograph \cite{peyreComputationalOptimalTransport2019} or the lecture notes \cite{nutzIntroductionEntropicOptimal} and references therein ).

As concerns multi-marginal optimal transport (MOT),  it arises naturally in many different areas of applications,  including economics 
\cite{carlier2010matching}, financial mathematics 
\cite{beiglbock2013model,dolinsky2014martingale,dolinsky2014robust, Ennajietal22}, statistics 
\cite{bigot2018characterization,carlier2016vector}, image
processing \cite{rabin2011wasserstein}, tomography 
\cite{abraham2017tomographic}, machine learning 
\cite{Hassleretal21, Trillosetal22}, fluid dynamics 
\cite{brenier1989least} and quantum physics and chemistry, in
the framework of density functional theory 
\cite{buttazzo2012optimal,cotarDensityFunctionalTheory2013,frieseckeStrongInteractionLimitDensity2023}.
The structure of solutions to the multi-marginal optimal transport problem is a notoriously delicate issue, and is
still not well understood, despite substantial efforts on the part of many researchers 
 \cite{gangboOptimalMapsMultidimensional1998, Carlier03, CarlierNazaret08, Heinich05, passUniquenessMongeSolutions2011,passLocalStructureOptimal2012, KimPass14,kimMultimarginalOptimalTransport2015, ColomboDePascaleDiMarino15, colomboCounterexamplesMultimarginalOptimal2016, PassVargas21, MoameniPass17, PassVargas21-2}; see also the surveys \cite{passMultimarginalOptimalTransport2015} and \cite{dimarinoOptimalTransportationTheory2017}. 
Since $\MOT_\eps$ can be seen a perturbation of $\MOT_0$, it is natural to study the behaviour as $\eps$ vanishes. In this paper we are mainly interested in investigating the rate of convergence of the entropic cost  $\MOT_\eps$ to $\MOT_0$ under some mild assumptions on the cost functions and marginals.

In particular we are going to extend the techniques introduced in \cite{carlierConvergenceRateGeneral2023} for two marginals to the multi-marginal case
which will also let us generalize the  bounds in \cite{carlierConvergenceRateGeneral2023} to the case of degenerate cost functions.  For the two marginals and non-degenerate case we also refer the reader to a very recent (and elegant) paper \cite{malamutConvergenceRatesRegularized2023} where the authors push a little further the analysis of the convergence rate by disentangling the roles of $\int c\dd\gamma$ and the relative entropy in the total cost and deriving convergence rate for both these terms. Notice that concerning the convergence rate of the entropic multi-marginal optimal transport an upper bound has been already established in \cite{ecksteinConvergenceRatesRegularized2023}, which depends on the number of marginals and the quantization dimension of the optimal solutions to \labelcref{MEOTpb} with $\eps = 0$. Here we provide an improved, smaller, upper bound, which will depend only on the marginals, but not on the optimal transport plans for the un-regularized problem, and we also provide a lower bound depending on a signature condition on the mixed second derivatives of the cost function, that was introduced in \cite{passLocalStructureOptimal2012}. The main difficulty consists in adapting the estimates of \cite{carlierConvergenceRateGeneral2023} to the local structure of the optimal plans described in \cite{passLocalStructureOptimal2012}.

Our main findings can be summarized as follows: we establish two
{\bfseries upper bounds}, one valid for locally Lipschitz costs and a finer one valid for locally semi-concave costs. The proofs rely, as in
\cite{carlierConvergenceRateGeneral2023}, on a multi-marginal variant of the 
block approximation introduced in
\cite{carlierConvergenceEntropicSchemes2017}. Notice that in this case the bound
will depend only on the dimension of the
support of the marginals. Moreover, for locally semi-concave cost
functions, by exploiting Alexandrov-type results as in 
\cite{carlierConvergenceRateGeneral2023}, we improve the upper bound by a 
$1/2$ factor, obtaining the following inequality for some $C^*\in\R_{+}$ 
\begin{equation}
    \boxed{\MOT_\eps \leq \MOT_0 + \frac 12 \Biggl(\sum_{i=1}^m d_i-\max_{1\leq i\leq m}d_i \Biggr)\eps \log(1/\eps) + C^*\eps.}
\end{equation}
We stress that this upper bound is smaller or equal than the one provided in \cite[Theorem~3.8]{ecksteinConvergenceRatesRegularized2023}, which is of the form $\frac 12 (m-1)D \eps\log(1/\eps) + O(\eps)$ where $D$ is a quantization dimension of the support of an optimal transport plan. Thus $D$ must be greater or equal than the maximum dimension of the support of the marginals, and of course $\sum_{i=1}^m d_i - \max_{1\leq i \leq m} d_i \leq (m-1) \max_{1\leq i\leq m} d_i$. The inequality may be strict for example in the two marginals case with unequal dimension, as shown in \Cref{sec:examples}.

For the {\bfseries lower bound}, from the dual formulation of \labelcref{MEOTpb} we have
\[ \MOT_\eps \geq \MOT_0-\eps\log\int_{\prod_{i=1}^mX_i}e^{-\frac{E(x_1,\ldots,x_m)}{\eps}}\dd\otimes_{i=1}^m\mu_i(x_i), \]
where $E(x_1,\ldots,x_m)=c(x_1,\ldots,x_m)-
\oplus_{i=1}^m\phi_i(x_i)$ is the duality gap and 
$(\phi_1,\ldots,\phi_m)$ are Kantorovich potentials for the un-regularized problem \labelcref{MEOTpb} with $\eps = 0$.  By using the singular values 
decomposition of the bilinear form obtained as an average of mixed second derivatives of the cost and a signature 
condition introduced in 
\cite{passUniquenessMongeSolutions2011}, we are able to 
prove that $E$ detaches quadratically from the set $\{E=0\}$ and this allows us to estimate  the previous 
integral in the desired way as in 
\cite{carlierConvergenceRateGeneral2023} and improve the results in \cite{ecksteinConvergenceRatesRegularized2023} where only an upper bound depending on the quantization dimension of the solution to the un-regularized problem is provided. Moreover, this slightly more flexible use of Minty's trick compared to \cite{carlierConvergenceRateGeneral2023} allows us to obtain a lower bound also for degenerate cost functions in the two marginals setting. Given a $\kappa$ depending on a signature condition (see \labelcref{signature-condition}) on the second mixed derivatives of the cost, the lower bound can be summarized as follows
\begin{equation}
\boxed{\MOT_\eps \geq \MOT_0 + \frac \kappa 2 \eps\log(1/\eps) - C_* \eps.}
\end{equation}
for some $C_*\in\R_+$.
\medskip

The paper is organized as follows: in 
\Cref{sec:prel} we recall the multi-marginal optimal 
transport problem, some results concerning the
structure of the optimal solution, in particular the ones 
in \cite{passUniquenessMongeSolutions2011}, and define its entropy regularization. \Cref{sec:upper_bound} is devoted to the upper bounds stated in \Cref{value_general_upper_bound} and \Cref{upper_bound_refined}. In \Cref{sec:lower_bound} we establish the lower bound stated in \Cref{prop:lower-bound}. Finally, in \Cref{sec:examples} we provide some examples for which we can get  the matching bounds.

\section{Preliminaries}
\label{sec:prel}

Given $m$ probability compactly supported measures $\mu_i$ on sub-manifolds $X_i$ of dimension $d_i$ in $\mathbb{R}^N$ for $i\in\{1,\ldots,m\}$ and a continuous cost function $c: X_1 \times X_2 \times \ldots\times X_m \rightarrow \mathbb{R}_+$, the multi-marginal optimal transport problem consists in solving the following optimization problem
\begin{equation}
\tag{MOT}
\label{pb:mmot}
\MOT_0\coloneqq\inf_{\gamma\in\Pi(\mu_1,\ldots,\mu_m)}  \int_{\mbf X} c(x_1,\ldots,x_m)\dd\gamma
\end{equation}
where $\mbf X\coloneqq X_1 \times X_2 \times \ldots\times X_m$ and $\Pi(\mu_1,\ldots,\mu_m)$ denotes the set of probability measures on $\mbf X$ whose marginals are the $\mu_i$.  The formulation above is also known as the \emph{Kantorovich problem} and it amounts to a linear minimization problem over a convex, weakly compact set; it is then not difficult to prove the existence of a solution by the direct method of calculus of variations. Much of the attention in the optimal transport community is  rather focused on uniqueness and the structure of the minimizers. In particular, one is mainly interested in determining if the solution is concentrated on the graph of a function $(T_2,\ldots,T_m)$ over the first marginal, where $(T_i)_\sharp\mu_1=\mu_i$ for $i\in\{1,\ldots,m\}$, in which case this function induces a solution \textit{à la Monge}, that is $\gamma=(\id,T_2,\ldots,T_m)_\sharp\mu_1$.

In the two marginals setting, the theory is fairly well understood and it is well-known that under mild conditions on the cost function (e.g. twist condition) and marginals (e.g. being absolutely continuous with respect to Lebesgue), the solution to \labelcref{pb:mmot} is unique and is concentrated on the graph of a function ; we refer the reader to \cite{santambrogioOptimalTransportApplied2015} to have glimpse of it.
The extension to the multi-marginal case is still not well understood, but it has 
attracted recently a lot of attention due
to a diverse variety of applications.

In
particular in his seminal works \cite{passUniquenessMongeSolutions2011,passLocalStructureOptimal2012} Pass established some conditions, more restrictive than in the two marginals case, to ensure the existence of a solution concentrated on a graph.
In this work we rely on the following (local) result in \cite{passLocalStructureOptimal2012} giving an upper bound on the dimension of the support of the solution to \labelcref{pb:mmot}.
Let $P$ be the set of partitions of $\{1,\ldots,m\}$ into two non-empty disjoint subsets: $p=\{p_-,p_+\}\in P$ if $p_-\bigcup p_+=\{1,\ldots,m\}$, $p_-\bigcap p_+=\emptyset$ and $p_-,p_+\neq\emptyset$. Then for each $p\in P$ we denote by $g_p$ the bilinear form on the tangent bundle $T\mbf X$
\[ g_p \coloneqq D^2_{p_- p_+} c+ D^2_{p_+ p_-}c\quad\text{where}\quad D^2_{p q} c \coloneqq \sum_{i\in p,j\in q}D^2_{x_i x_j}c\]
for every $p,q \subseteq \{1,\ldots,m\}$, and $D^2_{x_ix_j} c\coloneqq \sum_{\alpha_i,\alpha_j} \frac{\partial^2 c}{{x_i^{\alpha_i} x_k^{\alpha_k}}} \dd x_i^{\alpha_i} \otimes \dd x_j^{\alpha_j}$, defined for every $i,j$ on the whole tangent bundle $T \mathbf{X}$. Define
\begin{equation}\label{def_convex_bilin}
G_c\coloneqq\Biggl\{\sum_{p\in P}t_pg_p\;|\;(t_p)_{p\in P}\in\Delta_P\Biggr\}
\end{equation}
to be the convex hull generated by the $g_p$, then it is easy to verify that each $g\in G_c$ is symmetric and therefore its signature, denoted by $(d^+(g),d^-(g),d^0(g))$, is well defined.
Then, the following result from \cite{passLocalStructureOptimal2012} gives a control on the dimension of the support of the optimizer(s) in terms of these signatures.
\begin{theorem}[Part of {\cite[Theorem~2.3]{passLocalStructureOptimal2012}}]
\label{thm:bound_support}
Let $\gamma$ a solution to \labelcref{pb:mmot} and suppose that the signature of some $g\in G_c$ at a point $\mbf x\in \mbf X $ is $(d^+,d^-,d^0)$, that is the number of positive, negative and zero eigenvalues. Then, there exists a neighbourhood $N_{\mbf x}$ of $\mbf x$ such that $N_{\mbf x}\bigcap \spt{\gamma}$ is contained in a Lipschitz sub-manifold of $\mbf{X}$ with dimension no greater than $\sum_{i=1}^m d_i-d^{+}$.
\end{theorem}
\begin{remark}
For the following it is important to notice that by standard linear algebra arguments we  have for each $g\in G_c$ that $d^+(g)\leq \sum_{i=1}^m d_i-\max_id_i$. This implies that the smallest bound on the dimension of $\spt{\gamma}$ which \Cref{thm:bound_support} can provide is $\max_i d_i$. 
\end{remark}
\begin{remark}[Two marginals case]
When $m=2$, the only $g\in G_c$ coincides precisely with the pseudo-metric introduced by Kim and McCann in \cite{kimContinuityCurvatureGeneral2010}. Assuming for simplicity that $d_1=d_2=d$, they noted that $g$ has signature $(d,d,0)$ whenever $c$ is non-degenerate so  \Cref{thm:bound_support} generalizes their result since it applies even when non-degeneracy fails providing new information in the two marginals case: the signature of $g$ is $(r,r,2d-2r)$ where $r$ is the rank of $D^2_{x_1 x_2}c$. Notice that this will help us to generalize the results established in \cite{carlierConvergenceRateGeneral2023,ecksteinConvergenceRatesRegularized2023} to the case of a degenerate cost function.
\end{remark}
It is well known that under some mild assumptions the Kantorovich problem \labelcref{pb:mmot} is dual to the following 
\begin{equation}
\tag{MD}
\label{pb:mmot_dual}
    \sup\left\{  \sum_{i=1}^m\int_{X_i}\phi_i(x_i)\dd\mu_i \;|\;\phi_i \in \contspace_{b}(X_i),\text{ }\sum_{i=1}^m\phi_i(x^i) \leq c(x_1,\ldots,x_m)\right \}.
\end{equation}
Besides, it admits solutions $(\phi_i)_{1\leq i \leq m}$, called \emph{Kantorovich potentials}, when $c$ is continuous and all the $X_i$'s are compact, and these solutions may be assumed \emph{$c$-conjugate}, in the sense that for every $i \in \{1,\ldots,m\}$
\begin{equation}\label{c_conjugate}
\forall x\in X_i, \quad \phi_i(x) = \inf_{(x_j)_{j\neq i} \in \mbf X_{-i}} c(x_1,\ldots, x_{i-1},x,x_{i+1}, \ldots) - \sum_{1\leq j\leq m, j\neq i} \phi_j(x_j).
\end{equation}
We recall the entropic counterpart of \labelcref{pb:mmot}: given $m$ probability measures $\mu_i$ on  $X_i$ as before, and a continuous cost function $c: \mbf X \rightarrow \mathbb{R}_+$ the $\MOT_\eps$ problem is

\begin{equation}\tag*{\labelcref{MEOTpb}}
\MOT_\eps = \inf\left\{ \int_{X_1 \times \ldots \times X_m} c\dd\gamma + \eps \entropy(\gamma | \otimes_{i=1}^m \mu_i)\;|\;\gamma \in \Pi(\mu_1, \ldots, \mu_m)\right\},
\end{equation}
where $\entropy(\cdot|\otimes_{i=1}^m \mu_i)$ is the Boltzmann-Shannon relative entropy (or Kullback-Leibler divergence) w.r.t.\ the product measure $\otimes_{i=1}^m \mu_i$, defined for general probability measures $p,q$ as
\[
\entropy(p \,|\, q) = 
\begin{dcases*}
\displaystyle{\int_{\R^d} \rho \log(\rho)\, \dd q} & if $p = \rho q$,\\
+\infty & otherwise.
\end{dcases*}
\]
The fact that $q$ is a probability measure ensures that $\entropy(p \,|\, q) \geq 0$.
The dual problem of \labelcref{MEOTpb} reads as
\begin{equation}
\tag{MD$_\eps$}
    \label{pb:dual_meot}
    \MOT_\eps = \eps+ \sup\left\{  \sum_{i=1}^m\int_{X_i}\phi_i(x_i)\dd\mu_i -\eps\int_{\mbf X}e^{\frac{\sum_{i=1}^m\phi_i(x_i)-c(\mbf x)}{\eps}}\dd\otimes_{i=1}^m \mu_i\;|\;\phi_i \in \contspace_{b}(X_i)\right \},
\end{equation}
which is invariant by $(\phi_1,\ldots,\phi_m)\mapsto(\phi_1+\lambda_1,\ldots,\phi_m+\lambda_m)$ where $(\lambda_1,\ldots,\lambda_m)\in\R^m$ and $\sum_{i=1}^m\lambda_i=0$, see \cite{leonardSurveySchrodingerProblem2014,nutzStabilitySchrodingerPotentials2022,marinoOptimalTransportApproach2020} for some recent presentations. It admits an equivalent \enquote{log-sum-exp} form:

\begin{equation}\tag{MD$_\eps'$}\label{pb:dual_meot_lse}
    \MOT_\eps = \sup\left\{  \sum_{i=1}^m\int_{X_i}\phi_i(x_i)\dd\mu_i -\eps\log\left(\int_{\mbf X}e^{\frac{\sum_{i=1}^m\phi_i(x_i)-c(\mbf x)}{\eps}}\dd\otimes_{i=1}^m \mu_i\right)\;|\;\phi_i \in \contspace_{b}(X_i)\right \},
\end{equation}
which is invariant by the same transformations without assuming $\sum_{i=1}^m\lambda_i=0$.

From \labelcref{MEOTpb} and \labelcref{pb:dual_meot} we recover, as $\eps\to 0$, the unregularized multi-marginal optimal transport \labelcref{pb:mmot} and its dual \labelcref{pb:mmot_dual} we have introduced above.
The link between   multi-marginal optimal transport  and its  entropic regularization is very strong and a consequence of the $\Gamma-$convergence of \labelcref{MEOTpb} towards \labelcref{pb:mmot} (one can adapt the proof in  \cite{carlierConvergenceEntropicSchemes2017} or see \cite{benamouGeneralizedIncompressibleFlows2019,gerolinMultimarginalEntropytransportRepulsive2020} for $\Gamma-$convergence in some specific cases) is that
\[\lim_{\eps\to 0} \MOT_{\eps}=\MOT_0.\]

By the direct method in the calculus of variations and strict convexity of the entropy, one can show that \labelcref{MEOTpb}  admits a unique solution $\gamma_\eps$, called \emph{optimal entropic plan}. 
Moreover,  there exist $m$ real-valued Borel functions $\phi^\eps_i$ such that
\begin{equation}\label{eq:structure}
\gamma_\eps = \exp\Bigg(\frac{ \oplus_{i=1}^m \phi^\eps_i - c}{\eps}\Bigg)  \otimes_{i=1}^m \mu_i,
\end{equation}
where $\oplus_{i=1}^m \phi^\eps_i \coloneqq (x_1,\ldots,x_m) \mapsto \sum_{i=1}^m\phi_i^\eps(x_i)$, and in particular we have that
\begin{equation}
\label{eq:veps-rewritten}
\MOT_\eps = \sum_{i=1}^m \int_{X_i}\phi^\eps_i\,\dd\mu_i 
\end{equation}
and these functions have continuous representatives and are uniquely determined up a.e.\ to additive constants.
The reader is referred to the analysis of \cite{marinoOptimalTransportApproach2020}, to \cite{nennaNumericalMethodsMultiMarginal2016} for  the extension to the multi-marginal setting, and to \cite{borweinDecompositionMultivariateFunctions1992, borweinEntropyMinimizationDAD1994, csiszarDivergenceGeometryProbability1975, follmerEntropyMinimizationSchrodinger1997, ruschendorfClosednessSumSpaces1998} for earlier references on the two marginals framework.

The functions $\phi^\eps_i$ in \labelcref{eq:structure} are called \emph{Schr\"odinger potentials}, the terminology being motivated by the fact that they solve the dual problem \labelcref{pb:dual_meot} and are as such the (unique) solutions to the so-called \emph{Schr\"odinger system}: for all $i\in \{1,\ldots, m\}$,
\begin{equation}\label{eq:Ssystem}
\phi_i(x_i) = -\eps\log\int_{\mbf X_{-i}}e^{\frac{\oplus_{1\leq j\leq m, j\neq i} \phi^\eps_j-c(\mbf x)}{\eps}}\,\dd\otimes_{1\leq j\leq m, j\neq i} \mu_j\quad \textrm{for }\mu_i\textrm{-a.e.\ }x_i,
\end{equation}
where $\mbf X_{-i} = \prod_{1\leq j\leq m,j\neq i}^m X_j$. Note that \labelcref{eq:Ssystem} is a \enquote{softmin} version of the multi-marginal $c$-conjugacy relation for Kantorovich potentials.

\section{Upper bounds}
\label{sec:upper_bound}
We start by establishing an upper bound, which will depend on the dimension of the marginals, for locally Lipschitz cost functions. We will then improve it for locally semi-concave (in particular $\contspace^{2}$) cost functions.

\subsection{Upper bound for locally Lipschitz costs}
\label{subsec:upper_bound}
The natural notion of dimension which arises is the \emph{entropy dimension}, also called \emph{information dimension} or \emph{R\'enyi dimension} \cite{renyiDimensionEntropyProbability1959}.
\begin{definition}[R\'enyi dimension (following \cite{youngDimensionEntropyLyapunov1982})]
If $\mu$ is a probability measure over a metric space $X$, we set for every  $\delta > 0$,
\[H_\delta(\mu) = \inf \left\{ \sum_{n\in\N} \mu(A_n) \log(1/\mu(A_n)) \;|\; \forall n, \diam(A_n) \leq \delta, \text{ and } X = \bigsqcup_{n\in \N} A_n\right\},\]
where the infimum is taken over countable partitions $(A_n)_{n\in\N}$ of $X$ by Borel subsets of diameter less than $\delta$, and we define the \emph{lower and upper entropy dimension of $\mu$} respectively by:
\[
\underline \dim_R(\mu) \coloneqq \liminf_{\delta \to 0^+} \frac{H_\delta(\mu)}{\log(1/\delta)}, \quad \overline \dim_R(\mu) \coloneqq \limsup_{\delta \to 0^+} \frac{H_\delta(\mu)}{\log(1/\delta)}.
\]    
\end{definition}

Notice that if $\mu$ is compactly supported on a Lipschitz manifold of dimension $d$, then $N_\delta(\spt{\mu}) \leq d\log(1/\delta) + C$ for some constant $C > 0$ and $\delta \in (0,1]$, where $N_\delta(\spt{\mu})$ is the box-counting number of $\spt{\mu}$, i.e. the minimal number of sets of diameter $\delta > 0$ which cover $\spt{\mu}$. In particular, by concavity of $t\mapsto t\log(1/t)$, we have
\begin{equation}\label{upper_bound_box_counting_number}
H_\delta(\mu) \leq \log N_\delta(\spt{\mu}).
\end{equation}
We refer to the beginning of \cite[\S3.1]{carlierConvergenceRateGeneral2023} for additional information and references on R\'enyi dimension.

The following theorem establishes an upper bound for locally Lipschitz costs.

\begin{theorem}\label{value_general_upper_bound}
Assume that for $i \in \{1,\ldots,m\}$, $\mu_i \in \probspace(X_i)$ is a compactly supported measure on a Lipschitz sub-manifold $X_i$ of dimension $d_i$ and $c\in\contspace^{0,1}_\loc(\mbf X)$, then 
\begin{equation}
    \boxed{\MOT_\eps \leq \MOT_0 + \Biggl(\sum_{i=1}^m d_i-\max_{j\in\{1,\ldots,m\}}d_j \Biggr)\eps \log(1/\eps) + O(\eps).}
\end{equation}  
\end{theorem}

\begin{proof}
  Given an optimal plan $\gamma_0$ for $\MOT_0$, we use the so-called \enquote{block approximation} introduced in \cite{carlierConvergenceEntropicSchemes2017}. For every $\delta > 0$ and $i \in \{1,\ldots,m\}$, consider a partition $X_i = \bigsqcup_{n\in\N} A^n_{i}$ of Borel sets such that\footnote{We always consider the Euclidean distance over $\R^N$, but since the supports of the measures are compact and the sub-manifolds are Lipschitz, we may equivalently consider the intrinsic metric over the sub-manifolds: they are equivalent distances at small scale, i.e.\ for $\abs{y-x} \leq \delta_0$ for some $\delta_0>0$.} $\diam(A^n_{i}) \leq \delta$ for every $n\in\N$, and set
\[\mu^n_{i} \coloneqq\begin{dcases*}
\frac {\mu_i \mres A^n_{i}}{\mu_i(A^n_{i})}& if $\mu_i(A^n_{i}) > 0$,\\
0& otherwise,
\end{dcases*}\]
then for every $m$-uple $n = (n_1,\ldots,n_m)\in \N^m$, 
\[(\gamma_0)^n \coloneqq \gamma_0(A^{n_1}_{1} \times \ldots \times A^{n_m}_{m}) \mu^{n_1}_{1} \otimes\ldots\otimes \mu_{m}^{n_m},\]
and finally,
\[\gamma_\delta \coloneqq \sum_{n\in \N^m} (\gamma_0)^n.\]
By definition, $\gamma_\delta \ll \otimes_{i=1}^m \mu_i$ and we may check that its marginals are the $\mu_i$'s.
Besides, $\gamma_\delta(\mbf A) = \gamma_0(\mbf A)$ for every $\mbf A = \prod_{i=1}^m A_i^{n_i}$ where $n \in \N^m$, and for $\otimes_{i=1}^m \mu_i$-almost every $\mbf x = (x_1,\ldots,x_m) \in \prod_{i=1}^m A^{n_i}_i$,
\[\frac{\dd \gamma_\delta}{\dd\otimes_{i=1}^m \mu_i}(x_1,\ldots,x_m) \coloneqq\begin{dcases*} \frac{\gamma_0(A^{n_1}_{1} \times\ldots\times A^{n_m}_{m})}{\mu_1(A^{n_1}_{1}) \ldots \mu_m(A^{n_m}_{m})}& if $\mu_1(A^{n_1}_{1}) \ldots\mu_m(A^{n_m}_{m}) > 0$\\
0& otherwise. \end{dcases*}\]

Let us compute its entropy and assume for simplicity that the measure $\mu_m$ is the one such that $\overline\dim_R(\mu_m)=\max_{i\in\{1,\ldots,m\}}\overline \dim(\mu_i)$:
\begin{align*}
\entropy(\gamma_\delta \,|\, \otimes_{i=1}^m \mu_i) &= \sum_{n\in \N^m} \int_{ \prod_{i=1}^mA^{n_i}_{i}} \log\left(\frac{\gamma_0(A^{n_1}_{1} \times\ldots\times A^{n_m}_{m})}{\mu_1(A^{n_1}_{1}) \ldots \mu_m(A^{n_m}_{m})}\right) \dd \gamma_\delta\\
&= \sum_{n\in \N^m} \gamma_0(A^{n_1}_{1} \times\ldots\times A^{n_m}_{m}) \log\left(\frac{\gamma_0(A^{n_1}_{1} \times\ldots\times A^{n_m}_{m})}{\mu_1(A^{n_1}_{1}) \ldots \mu_m(A^{n_m}_{m})}\right)\\
&= \begin{multlined}[t]
\sum_{n\in \N^m} \gamma_0(A^{n_1}_{1} \times\ldots\times A^{n_m}_{m}) \log\left(\frac{\gamma_0(A^{n_1}_{1} \times\ldots\times A^{n_m}_{m})}{\mu_m(A^{n_m}_{m})}\right)\\
+ \sum_{j=1}^{m-1}\sum_{n\in\N^m} \gamma_0(A^{n_1}_{1} \times\ldots\times A^{n_m}_{m}) \log(1/\mu_j(A^{n_j}_j))
\end{multlined}\\
&= \begin{multlined}[t]
\sum_{n\in \N^m} \gamma_0(A^{n_1}_{1} \times\ldots\times A^{n_m}_{m}) \log\left(\frac{\gamma_0(A^{n_1}_{1} \times\ldots\times A^{n_m}_{m})}{\mu_m(A^{n_m}_{m})}\right)\\
+ \sum_{j=1}^{m-1}\sum_{n_j\in \N} \gamma_0\left(\prod_{i=1}^{j-1}X_i \times A_j^{n_j}\times \prod_{i=j+1}^m X_i\right)\mu_j(A^{n_j}_{j}) \log(1/\mu_j(A^{n_j}_{j}))
\end{multlined}\\
& \leq \sum_{j=1}^{m-1}\sum_{n_j\in \N}\mu_j(A^{n_j}_{j}) \log(1/\mu_j(A^{n_j}_{j})) .
\end{align*}
the last inequality coming from the inequality $\gamma_0(A^{n_1}_{1} \times\ldots\times A^{n_m}_{m}) \leq \mu_m(A^{n_m}_{m})$. Taking partitions $(A^n_{j})_{n\in\N}$ of diameter smaller than $\delta$ such that $\sum_{n_j\in \N} \mu_j(A_j^{n_j}) \log(1/\mu_j(A_j^{n_j})) \leq H_\delta(\mu_j) + \frac 1{m-1}$ we get
\[\entropy(\gamma_\delta \,|\, \otimes_{i=1}^m \mu_i) \leq \sum_{j=1}^{m-1} H_\delta(\mu_j)+1.\]

Since the $\mu_i$'s have compact support and $c$ is locally Lipschitz, for $\delta$ small enough there exists $L\in (0,+\infty)$ not depending on $\delta$ such that $[c]_{\contspace^{0,1}(\mbf A)} \leq L$ for every $\mbf A \in \mathcal A \coloneqq \{ \prod_{i=1}^m A_i^{n_i} \;|\; n\in \N^m, \mu_1(A_1^{n_1}) \ldots \mu_m(A_m^{n_m}) > 0\}$. Notice that the $\infty$-Wasserstein distance (see \cite[\S 3.2]{santambrogioOptimalTransportApplied2015}) with respect to the norm $\norm{\cdot}$\footnote{The Wasserstein distance of order $p$ is defined here by $W_p^p(\mu,\nu) \coloneqq \inf\{ \int \norm{y-x}^p \dd\gamma(x,y) \;|\; \gamma \in \Pi(\mu,\nu)\}$ for $p\in [1,+\infty)$ and by $W_\infty(\mu,\nu) = \inf \{ \gamma-\esssup \,(x,y) \mapsto \norm{y-x}\;|\;\gamma \in \Pi(\mu,\nu)\}$ for $p=+\infty$.} satisfies $W_\infty(\gamma_\delta,\gamma_0) \leq \delta$. Indeed, $\diam \mbf A \leq \delta$ and $\gamma_0(\mbf A) = \gamma_\delta(\mbf A)$ for every $\mbf A \in \mathcal A$, so that $\Gamma \coloneqq \sum_{\mbf A\in \mathcal A} \gamma_0(\mbf A)^{-1} (\gamma_\delta \mres \mbf A) \otimes (\gamma_0 \mres \mbf A)$ is a transport plan from $\gamma_\delta$ to $\gamma_0$ satisfying $\Gamma-\esssup\,(\mbf x,\mbf{x'}) \mapsto \norm{\mbf{x'}-\mbf{x}} \leq \delta$. Thus taking $\gamma_\delta$ as competitor in \labelcref{MEOTpb} we obtain:
\begin{equation}\label{upper_bound_eps_delta}
\begin{aligned}
\MOT_\eps &\leq \int c \dd\gamma_\delta + \eps \sum_{j\leq m-1} H_\delta(\mu_j) + \eps\\
&= \MOT_0 + \sum_{\mbf A \in \mathcal A}\int_{\mbf A} c \dd(\gamma_\delta - \gamma_0) + \eps \sum_{j\leq m-1}H_\delta(\mu_j)+ \eps \\
&\leq \MOT_0 + \sum_{\mbf A \in \mathcal A} L W_\infty(\gamma_\delta\mres \mbf A,\gamma_0\mres \mbf A)\gamma_0(\mbf A) + \eps \sum_{j\leq m-1}H_\delta(\mu_j)+ \eps
\\
&\leq \MOT_0 + L \delta  + \eps\sum_{j\leq m-1} \frac{H_\delta(\mu_j)}{\log(1/\delta)}\log(1/\delta)+ \eps.
\end{aligned}
\end{equation}
Taking $\delta = \eps$ and recalling that the $\mu_j$'s are concentrated on sub-manifolds of dimension $d_j$, which implies that $H_\delta(\mu_j) \leq d_j \log(1/\delta) + \frac{C^*-1-L}{m-1}$ for some $C^*\geq L+1$ and for every $j\in \{1,\ldots,m\}$, we get
\[\MOT_\eps \leq \MOT_0 + \Bigl(\sum_{j\leq m-1} d_j\Bigr) \eps\log(1/\eps) + C^*\eps.\qedhere\]  
\end{proof}
\begin{remark}
If the $\mu_i$'s are merely assumed to have compact support (not necessarily supported on a sub-manifold), the above proof actually shows the slightly weaker estimate
\begin{equation}
    {\MOT_\eps \leq \MOT_0 + \Biggl(\sum_{i=1}^m \overline\dim_R(\mu_i)-\max_{j\in\{1,\ldots,m\}} \overline\dim_R(\mu_j) \Biggr)\eps \log(1/\eps) + o(\eps\log(1/\eps).}
\end{equation}
Indeed, for every $i$, by definition of $\overline \dim_R(\mu_i) \eqqcolon d_i$ we have $\frac{H_\delta(\mu_i)}{\log(1/\delta)} \leq \sup_{0<\delta' \leq \delta } \frac{H_{\delta'}(\mu_i)}{\log(1/\delta')} = d_i + o(1)$ as $\delta \to 0$ thus taking $\delta =\eps$ as above we have $\eps \frac{H_\delta(\mu_i)}{\log(1/\delta)} \log(1/\delta) \leq (d_i+o(1))\eps\log(1/\eps)$.

Besides, notice that by taking $m=2$ and $d_1=d_2=d$, one easily retrieves \cite[Proposition~3.1]{carlierConvergenceRateGeneral2023}.
\end{remark}

\subsection{Upper bound for locally semi-concave costs}
\label{ssubec:upper_bound_refined}
We provide now a finer upper bound under the additional assumptions that the $X_i$'s are $\contspace^2$ sub-manifolds of $\R^N$, $c$ is locally semi-concave as in \Cref{definition_loc_semiconcave} (which is the case when $c\in \contspace^2(\mbf X,\R_+)$), and the $\mu_i$'s are measures in $L^\infty(\hdm^{d_i}_{X_i})$ with compact support in $X_i$.

\begin{definition}\label{definition_loc_semiconcave}
    A function $f : X \to \R$ defined on a $\contspace^2$ sub-manifold $X \subseteq \R^N$ of dimension $d$ is \emph{locally semiconcave} if for every $x\in X$ there exists a local chart (i.e. a $\contspace^2$ diffeomorphism) $\psi : U\to \Omega$ where $U \subseteq X$ is an open neighborhood of $x$ and $\Omega$ is an open convex subset of $\R^d$, such that $f\circ {\mbf \psi}^{-1}$ is $\lambda$-concave for some $\lambda \in \R$, meaning $f\circ {\mbf \psi}^{-1}- \lambda\frac{\abs{\cdot}^2}2$ is concave on $\Omega$.
\end{definition}

\begin{lemma}[Local semiconcavity and covering]\label{potentials_semiconcavity}
    Let $c : \mbf X \to \R_+$ be a locally semiconcave cost function and $(\phi_i)_{1\leq i \leq m} \in \prod_{\leq i\leq m} \contspace(K_i)$ be a system of $c$-conjugate functions as in \labelcref{c_conjugate} defined on compact subsets $K_i\subseteq X_i$. We can find $\lambda \in\R, J\in \N^*$ and for every $i\in\{ 1,\ldots, m\}$ a finite open covering $(U_i^j)_{1\leq j\leq J}$ of $K_i$ together with bi-Lipschitz local charts $\psi_i^j : U_i^j \to \Omega_i^j$ satisfying the following properties, having set $\mbf \Omega^{\mbf j} \coloneqq \prod_{1\leq i\leq m} \Omega_i^{j_i}$ and ${\mbf \psi}^{\mbf j} \coloneqq (\psi_1^{j_1},\ldots, \psi_m^{j_m})$ for every $\mbf j = (j_1, \ldots, j_m) \in \{1,\ldots,J\}^m$:
    \begin{enumerate}[(i)]
    \item for every $\mbf j \in \{1,\ldots,J\}^m$, $c \circ ({\mbf \psi}^{\mbf{j}})^{-1}$ is $\lambda$-concave on $\mbf\Omega^{\mbf j}$,
    \item for every $(i,j) \in \{1,\ldots, m\} \times\{1,\ldots, J\}$, $\phi_i \circ (\psi_i^j)^{-1}$ is $\lambda$-concave on $\Omega_i^j$.
    \end{enumerate} 
    In particular, all the $\phi_i$'s are locally semiconcave.
\end{lemma}
\begin{proof}
    For every $i$, by compactness of the $K_i$'s we can find a finite open covering $(U_i^j)_{1\leq j\leq J}$ of $K_i$ and bi-Lipschitz local charts $\psi_i^j : U_i^j \to \Omega_i^j$ such that for every $\mbf j = (j_1, \ldots, j_m) \in \{1,\ldots, J\}^m$, $c \circ ({\mbf \psi}^{\mbf{j}})^{-1} - \lambda^{\mbf{j}}\frac{\abs{\cdot}^2}2$ is concave for some $\lambda^{\mbf{j}} \in \R$. We may assume that $\lambda^{\mbf j} = \lambda$ for every $\mbf{j}$, by taking $\lambda \coloneqq \max \{\lambda^{\mbf j}\;|\;\mbf j\in \{1,\ldots,J\}^m\}$. Fix $i \in\{1,\ldots,m\}, j \in\{1,\ldots,J\}$, then for every $\mbf k = (k_{\ell})_{\ell\neq i} \in \{1,\ldots,J\}^{m-1}$ set $\mbf{\hat k} = (k_1, \ldots, k_{i-1},j,k_{i+1}, \ldots)$. Notice that for every $y\in \Omega_i^j$,
    \begin{align*}
    \quad&\phi_i\circ(\psi_i^j)^{-1}(y)\\
    &= \inf_{(x_\ell)_{\ell\neq i} \in \mbf{K}_{-i}} c(x_1,\ldots, x_{i-1},(\psi_i^j)^{-1}(y),x_{i+1}, \ldots) - \sum_{\ell : \ell\neq i} \phi_\ell(x_\ell)\\
    &= \min_{\mbf k = (k_\ell)_{\ell\neq i}} \inf_{(y_\ell)_{\ell\neq i}\in \mbf{\Omega^{\hat k}}_{-i}}   c\circ ({\mbf \psi}^{\mbf{\hat k}})^{-1}(y_1, \ldots, y_{i-1},y,y_{i+1},\ldots) - \sum_{\ell : \ell\neq i} \phi_\ell\circ (\psi_\ell^{k_\ell})^{-1}(y_\ell),
    \end{align*}
    and we see that it is  $\lambda$-concave as an infimum of $\lambda$-concave functions.
\end{proof}

We are going to use an integral variant of Alexandrov Theorem which is proved in \cite{carlierConvergenceRateGeneral2023}.
\begin{lemma}[{\cite[Lemma~3.6]{carlierConvergenceRateGeneral2023}}] \label{lem:alexandrov-modified}
Let $f : \Omega \to \R$ be a $\lambda$-concave function defined on a convex open set $\Omega \subseteq \R^d$, for some $\lambda \geq 0$. There exists a constant $C \geq 0$ depending only on $d$ such that:
\begin{equation}\label{integral_lambda_convexity_Taylor_inequality}
\int_{\Omega} \sup_{y\in B_r(x)\cap \Omega} \abs{f(y) - (f(x)+ \nabla f(x) \cdot (y-x))} \dd x \leq C r^2 \hdm^{d-1}(\partial \Omega) ([f]_{\contspace^{0,1}(\Omega)} + \lambda \diam(\Omega)).
\end{equation}
\end{lemma}

We may now state the main result of this section.

\begin{theorem}
\label{upper_bound_refined}
    Let $c\in\contspace^{2}(\mbf X)$ and assume that for every $i \in \{1,\ldots, m\}$, $X_i \subseteq \R^N$ is a $\contspace^2$ sub-manifold of dimension $d_i$ and $\mu_i \in L^\infty(\hdm^{d_i}_{X_i})$ is a probability measure compactly supported in $X_i$. Then there exists constants $\eps_0, C^*\geq 0$ such that for $\eps \in (0,\eps_0]$
    \begin{equation}
    \boxed{\MOT_\eps \leq \MOT_0 + \frac 12 \Biggl(\sum_{i=1}^m d_i-\max_{1\leq i\leq m}d_i \Biggr)\eps \log(1/\eps) + C^*\eps.}
\end{equation}
\end{theorem}


\begin{proof}
The measures $\mu_i$ being compactly supported in $X_i$, take for every $i \in \{1,\ldots,m\}$ an open subset $U_i$ of $X_i$ such that $\spt{\mu_i} \subseteq U_i \Subset X_i$ and define the compact set $K_i \coloneqq\bar U_i$. Take $(\phi_i)_{1\leq i \leq m} \in \prod_{1\leq i\leq m} \contspace(K_i)$ a $m$-uple of $c$-conjugate Kantorovich potentials and a transport plan $\gamma_0\in\Pi(\mu_1,\ldots,\mu_m)$ which are optimal for the unregularized problems \labelcref{pb:mmot_dual} and \labelcref{pb:mmot} respectively. In particular,
\begin{equation}\label{E_minimal}
\begin{aligned}
    E \coloneqq c -\oplus_{i=1}^{m}\phi_i &\geq 0&
    \text{on}&\quad \mbf U = \prod_{1\leq i\leq m} U_i,\\
    E &= 0&
    \text{on}&\quad \spt{\gamma_0} \subseteq \mbf U.
\end{aligned}
\end{equation}
For every $i \in \{1,\ldots,m\}$ we consider the coverings $(U_i^j)_{1\leq j\leq J}$ and bi-Lipschitz local charts $\psi_i^j : U_i^j \to \Omega_i^j$ for $j\in \{1,\ldots, J\}$ provided by \Cref{potentials_semiconcavity} and we notice by compactness that there exists open subsets $\tilde U_i^j \Subset U_i^j$ such that for a small $\delta_0 > 0$, the $\delta_0$-neighbourhood of $\tilde \Omega_i^j \coloneqq \psi_i^j(\tilde U_i^j)$ is included in $\Omega_i^j$ for every $j$, and $(\tilde U_i^j)_{1\leq j\leq J}$ is still an open covering of $K_i$. For $\delta \in (0,\delta_0)$ we consider the block approximation $\gamma_\delta$ of $\gamma_0$ built in the proof of \Cref{value_general_upper_bound}, as well as some $\kappa_\delta \in \Pi(\gamma_0,\gamma_\delta)$ such that $\sup_{(\mbf{x_0}, \mbf{x})\in \spt{\kappa_\delta}} \norm{\mbf{x_0}-\mbf{x}}\leq \delta$. For every $\mbf j = (j_1,\ldots,j_m) \in \{1,\ldots,J\}^m$, we set $E^{\mbf j} \coloneqq E \circ ({\mbf \psi}^{\mbf j})^{-1}$, $U^{\mbf j} \coloneqq \prod_{i=1}^m U_i^{j_i}$ and $\tilde U^{\mbf j} \coloneqq \prod_{i=1}^m \tilde U_i^{j_i}$, and we write:
\begin{align*}
    \int_{\mbf X} c \dd\gamma_\delta - \int_{\mbf X} c \dd\gamma_0 &= \int_{\mbf U} E \dd\gamma_\delta\\
    &= \int_{\mbf U\times \mbf U} E(\mbf{x}) \dd \kappa_\delta(\mbf{x_0},\mbf{x})\\
    &\leq \sum_{\mbf j \in \{1,\ldots,J\}^m} \int_{(\mbf{x_0}, \mbf{x}) \in\tilde U^{\mbf{j}} \times \mbf{U}} E(\mbf{x}) \dd \kappa_\delta(\mbf{x_0},\mbf{x})\\
    &\leq \sum_{\mbf j \in \{1,\ldots,J\}^m} \int_{(\mbf{x_0}, \mbf{x}) \in (U^{\mbf{j}})^2} E^{\mbf j}({\mbf \psi}^{\mbf j}(\mbf{x})) \dd \kappa_\delta(\mbf{x_0},\mbf{x}).
    \end{align*}
Notice that for every $\mbf j \in \{1,\ldots,J\}^m$ and $\gamma_0$-a.e.\ $\mbf{x_0} \in U^{\mbf j}$, $E^{\mbf j}$ is differentiable at ${\mbf \psi}^{\mbf j}(\mbf{x_0})$, or equivalently $E$ is differentiable at $\mbf{x_0}$. Indeed $c$ is differentiable everywhere, and for every $i \in \{1,\ldots,m\}$ and $j \in\{1,\ldots,J\}$, $\phi_i \circ (\psi_i^j)^{-1}$ is semi-concave thus differentiable $\lbm^{d_i}$-a.e.\ hence $\phi_i$ is differentiable $\mu_i$-a.e.\ on $U_i^j$ because $\mu_i \ll \hdm^{d_i}$ and $\psi_i^j$ is bi-Lipschitz, which in turn implies that $\oplus_{i=1}^m \phi_i$ is differentiable $\gamma_0$-a.e.\ on $U^{\mbf j}$ because $\gamma_0 \in \Pi(\mu_1,\ldots, \mu_m)$. Moreover, by \labelcref{E_minimal} we have $T_{{\mbf \psi}^{\mbf j}(\mbf{x_0})} E^{\mbf j} \equiv 0$ for $\gamma_0$-a.e.\ $\mbf{x_0} \in U^{\mbf j}$, where $T_{y_0} f$ designates the first order Taylor expansion $y \mapsto f(y_0) + \nabla f(y_0)\cdot (y-y_0)$ for any function $f$ which is differentiable at $y_0$. We may then compute:
    \begin{equation}\label{upper_ineq_cost}
        \begin{aligned}
        &\qquad \int_{\mbf X} c \dd\gamma_\delta - \int_{\mbf X} c \dd\gamma_0\\
        &\leq \sum_{\mbf j \in \{1,\ldots, J\}^m} \int_{(\mbf{x_0}, \mbf{x}) \in (U^{\mbf{j}})^2} \Bigl(E^{\mbf j}({\mbf \psi}^{\mbf j}(\mbf{x}))-T_{{\mbf \psi}^{\mbf j}(\mbf{x_0})} E^{\mbf j}\bigl({\mbf \psi}^{\mbf j}(\mbf{x})-{\mbf \psi}^{\mbf j}(\mbf{x_0})\bigr)\Bigr) \dd \kappa_\delta(\mbf{x_0},\mbf{x})\\
        &= \sum_{\mbf j = (j_1,\ldots, j_m)} \left(\int_{(\mbf{x_0}, \mbf{x}) \in (U^{\mbf{j}})^2} \Bigl(c^{\mbf j}({\mbf \psi}^{\mbf j}(\mbf{x}))-T_{{\mbf \psi}^{\mbf j}(\mbf{x_0})} c^{\mbf j}\bigl({\mbf \psi}^{\mbf j}(\mbf{x})-{\mbf \psi}^{\mbf j}(\mbf{x_0})\bigr)\Bigr) \dd \kappa_\delta(\mbf{x_0},\mbf{x}) \right.\\
        &\quad- \left.\sum_{i=1}^m \int_{(x_0, x) \in (U_i^{j_i})^2} \Bigl(\phi_i^{j_i}(\psi_i^{j_i}(x))-T_{\psi_i^{j_i}(x_0)} \phi_i^{j_i}\bigl(\psi_i^{j_i}(x)-\psi_i^{j_i}(x_0)\bigr)\Bigr)\dd (e_i,e_i)_\sharp \kappa_\delta(x_0,x)\right).
            \end{aligned}
    \end{equation}

Now, since $c^{\mbf j}$ is $\lambda$-concave on each $\Omega^{\mbf j} \coloneqq {\mbf \psi}^{\mbf j}(U^{\mbf j})$, whenever $\norm{\mbf{x_0}-\mbf{x}} \leq \delta$ we have
\begin{equation}\label{ineq_c}
c^{\mbf j}({\mbf \psi}^{\mbf j}(\mbf{x}))-T_{{\mbf \psi}^{\mbf j}(\mbf{x_0})} c^{\mbf j}\bigl({\mbf \psi}^{\mbf j}(\mbf{x})-{\mbf \psi}^{\mbf j}(\mbf{x_0})\bigr) \leq \lambda \frac{\abs{{\mbf \psi}^{\mbf j}(\mbf{x})-{\mbf \psi}^{\mbf j}(\mbf{x_0})}^2}2\leq \frac{m\lambda L^{\mbf j}}2 \delta^2,
\end{equation}
where $L^{\mbf j} \coloneqq \max\{\lipconst({\mbf \psi}^{\mbf j}),\lipconst(({\mbf \psi}^{\mbf j})^{-1})\}$. Besides, we may apply \Cref{lem:alexandrov-modified} to each $\phi_i^{j_i}$ over $\Omega_i^{j_i}$ to get
\begin{equation}\label{ineq_phi}
\begin{aligned}
    \qquad& \abs*{\int_{(x_0, x) \in (U_i^{j_i})^2} \Bigl(\phi_i^{j_i}(\psi_i^{j_i}(x))-T_{\psi_i^{j_i}(x_0)} \phi_i^{j_i}\bigl(\psi_i^{j_i}(x)-\psi_i^{j_i}(x_0)\bigr)\Bigr)\dd (e_i,e_i)_\sharp \kappa_\delta(x_0,x)}\\
    \leq& \int_{x_0 \in U_i^{j_i}} \sup_{y \in B_{L^{\mbf j}\delta}({\psi_i^{j_i}(x_0)})\cap \Omega_i^{j_i}}\abs*{\phi_i^{j_i}(y)-T_{\psi_i^{j_i}(x_0)} \phi_i^{j_i}\bigl(y-\psi_i^{j_i}(x_0)\bigr)} \dd (e_i,e_i)_\sharp \kappa_\delta(x_0,x)\\
    =& \int_{U_i^{j_i}} \sup_{y \in B_{L^{\mbf j}\delta}({\psi_i^{j_i}(x_0)})\cap \Omega_i^{j_i}}\abs*{\phi_i^{j_i}(y)-T_{\psi_i^{j_i}(x_0)} \phi_i^{j_i}\bigl(y-\psi_i^{j_i}(x_0)\bigr)} \dd \mu_i(x_0)\\
    \leq& \int_{\Omega_i^{j_i}} \sup_{y \in B_{L^{\mbf j}\delta}(y_0)\cap \Omega_i^{j_i}}\abs*{\phi_i^{j_i}(y)-T_{y_0} \phi_i^{j_i}\bigl(y-y_0\bigr)} \dd (\psi_i^{j_i})_\sharp \mu_i(y_0)\\
    \leq & \norm{\mu_i}_{L^\infty(\hdm^{d_i})} L^{\mbf j} C (L^{\mbf j} \delta)^2 \hdm^{d_i-1}(\partial \Omega_i^{j_i}) \left([\phi_i^{j_i}]_{\contspace^{0,1}(\Omega_i^{j_i})} + \lambda \diam(\Omega_i^{j_i})\right)\\
    \leq& C^{\mbf j} \delta^2,
\end{aligned}
\end{equation}
for some constant $C^{\mbf j} \in (0,+\infty)$ which does not depend on $\delta$. Reporting \labelcref{ineq_c} and \labelcref{ineq_phi} in \labelcref{upper_ineq_cost} yields
\[\int_{\mbf X} c \dd\gamma_\delta - \int_{\mbf X} c \dd\gamma_0 \leq \sum_{\mbf{j}\in \{1,\ldots, J\}^m} \left(\frac{m\lambda L^{\mbf j}}2 + C^{\mbf j}\right) \delta^2 \eqqcolon C' \delta^2.\]
Finally, we proceed as in the end of the proof of \Cref{value_general_upper_bound}, taking $\gamma_\delta$ as competitor in the primal formulation \labelcref{MEOTpb}, so as to obtain
\begin{align*}
    \MOT_\eps - \MOT_0 &\leq \int_{\mbf X} c \dd\gamma_\delta - \int_{\mbf X} c \dd\gamma_0 + \eps \sum_{i\leq m-1} H_\delta(\mu_j) + \eps \\
    &\leq C' \delta^2 + \eps \sum_{i\leq m-1} (d_i \log(1/\delta) + C''),
\end{align*}
where $C'' \in (0,+\infty)$ is a constant such that $H_\delta(\mu_i) \leq d_i \log(1/\delta) + C'' -1$. Taking $\delta = \sqrt{\eps}$ for $\eps \leq \delta_0^2$ yields
\[\MOT_\eps -\MOT_0 \leq \frac 12 \left(\sum_{i=1}^{m-1} d_i\right)\eps\log(1/\eps) + (C'+(m-1)C'')\eps,\]
and we obtain the desired estimate recalling that the index $i=m$ was chosen merely to simplify notations.
\end{proof}

\section{Lower bound for \texorpdfstring{$\contspace^2$}{C2} costs with a signature condition}
\label{sec:lower_bound}

In this section we consider a cost $c \in \contspace^2(\mbf X,\R_+)$ where $\mbf X = X_1 \times \ldots \times X_m$ and we will assume that for every $i\in \{1\ldots,m\}$, the measure $\mu_i$ is compactly supported on a $\contspace^2$ sub-manifold $X_i \subseteq \R^N$ of dimension $d_i$. We are going to establish a a lower bound in the same form as the fine upper bound of \Cref{upper_bound_refined}, the dimensional constant being this time related to the signature of some bilinear forms, following ideas from \cite{passLocalStructureOptimal2012}.

\begin{lemma}\label{gap_lemma}
    Let $c \in \contspace^2(\mbf X, \R_+)$ and $(\phi_1,\ldots, \phi_m) \in \contspace(K_1)\times \cdots \contspace(K_m)$ be a system of $c$-conjugate functions on subsets $K_i \subseteq X_i$ for every $i$. We set $\ener \coloneqq c- \phi_1\oplus \ldots \oplus \phi_m$ on $\mbf K \coloneqq K_1 \times \ldots \times K_m$
and we take $\mbf{\bar x} \in \mbf K$ as well as some $g_{\mbf{\bar x}} \in \{g(\mbf{\bar x}) \;|\; g\in G_c\}$ of signature $(d^+,d^-,d^0)$, $G_c$ being defined in \labelcref{def_convex_bilin}. Then there exists local coordinates around $\mbf{\bar x}$, i.e. $\contspace^2$ diffeomorphisms
\[u = (u^0, u^-, u^+) : U \subseteq \mbf{X} \to B_\rho^{d^+}(0)\times B_\rho^{d^-}(0) \times B_\rho^{d^0}(0),\]
$U$ being an open neighborhood of $\mbf{\bar x}$, such that if $\mbf x, \mbf{x'} \in B_r(\mbf{\bar x}) \subseteq U$,
\begin{equation}\label{gap_inequality}
\frac{\ener(\mbf{x'}) + \ener(\mbf{x})}2 \geq  \abs{u^+(\mbf{x'})-u^+(\mbf x)}^2-\abs{u^-(\mbf{x'})-u^-(\mbf x)}^2 -\eta(r)\abs{u(\mbf{x'})-u(\mbf x)}^2
\end{equation}
where $\eta(r)\geq 0$ tends to $0$ as $r\to 0$.
\end{lemma}
\begin{proof}
    Let $p = \{p_-,p_+\} \in P$. For $y \in \prod_{i\in p_\pm} K_i$, we set
        \[\phi_{p_\pm}(y) \coloneqq \sum_{i\in p_\pm} \phi_i(y_i).\] 
    We identify any $\mbf x \in \mbf K$ with $(x_{p_-},x_{p_+})$. Since the $\phi_i$'s are c-conjugate, for $\mbf x, \mbf{x'} \in \mathbf K$ it holds:
    \begin{align*}
        E(\mbf{x'}) &= c(x'_{p_-}, x'_{p_+}) - \phi_{p_-}(x'_{p_-}) - \phi_{p_+}(x'_{p_+})\\
        &\geq c(x_{p_-}',x_{p_+}') - (c(x_{p_-}',x_{p_+})-\phi_{p_+}(x_{p_+}))- (c(x_{p_-},x_{p_+}')-\phi_{p_-}(x_{p_-}))\\
        &= c(x_{p_-}',x_{p_+}')-c(x_{p_-}',x_{p_+})-c(x_{p_-},x_{p_+}') + c(x_{p_-},x_{p_+}) - E(\mbf x).
    \end{align*}
Now we do computations in local charts $\psi_i : U_i \subseteq X_i \to \psi_i(U_i) \subseteq \R^{d_i}$ which are $\contspace^2$ diffeomorphisms such that $B_R(\bar x_i) \subseteq U_i$ for some $R > 0$ and $\psi_i(U_i)$ are balls centered at $0$ for every $i \in \{1,\ldots,m\}$. With a slight abuse, we use the same notation for points and functions written in these charts, and use Taylor's integral formula\footnote{Any linear combination $az_i +by_i$ will designate $\psi_i^{-1}(a\psi_i(z_i) +b \psi_i(y_i))$.}:
\begin{equation*}
E({\mbf x'}) + E({\mbf x}) \geq  \int_0^1 \int_0^1 D^2_{p_- p_+} c(\mbf{x}_{s,t})(x_{p_-}'-x_{p_-}, x_{p_+}'-x_{p_+}) \dd s \dd t
\end{equation*}
where $\mbf{x}_{s,t} \coloneqq (x_{p_-}+(1-s)x_{p_-}',x_{p_+}+(1-t)x_{p_+}')$ for $s,t\in [0,1]$. Since $\abs{D^2_{p_- p_+} c(\mbf x_{s,t})-D^2_{p_- p_+} c(\mbf{\bar x})} \leq \eta(r)$ where $\eta$ is the maximum for $p\in P$ of the moduli of continuity of $D^2_{p^- p^+} c$ at $\mbf{\bar x}$. Since $\eta$ is independent from $p$ and tends to $0$ as $r \to 0$ because $c$ is $\contspace^2$,  and by definition $D^2 c(\mbf {\bar x})(x_{p_-}'-x_{p_-},x_{p_+}'-x_{p_+}) = \frac 12 g_p(\mbf {\bar x})(\mbf{x'}-\mbf{x},\mbf{x'}-\mbf{x})$, it holds:
\begin{align}
    E(\mbf x) + E(\mbf {x'}) 
    &\geq \frac 12 g_p(\mbf{\bar x})(\mbf{x'}-\mbf{x},\mbf{x'}-\mbf{x}) - \eta(r) \norm{\mbf x'-\mbf x}^2.\notag
\end{align}
Taking $g_{\mbf{\bar x}} = \sum_{p\in P} t_p g_p(\mbf{\bar x})$ for some $(t_p)_{p\in P} \in \Delta_P$ and averaging the previous inequality yields:
\begin{equation}\label{lower_bound_energy_g}
    E(\mbf x) + E(\mbf {x'}) \geq \frac 12 g_{\mbf{\bar x}}(\mbf{x'}-\mbf x,\mbf{x'}-\mbf{x}) -\eta(r) \norm{\mbf{x'}-\mbf x}^2.
\end{equation}
Finally, we can find a linear isomorphism $Q \in GL(\sum_{i=1}^m d_i, \R)$ which diagonalizes $g_{\mbf{\bar x}}$, such that after setting $u \coloneqq Q\circ (\psi_1,\ldots, \psi_m)$ and denoting $u = (u^+,u^-,u^0) : \prod_{i=1}^m U_i \to \R^{d^+}\times \R^{d^-} \times \R^{d^0}$, where $(d^+,d^-,d^0)$ is the signature of $g_{\mbf{\bar x}}$, it holds:
\[ \frac 14 g_{\mbf{\bar x}}(\mbf{x'}-\mbf x,\mbf{x'}-\mbf x)  = \abs{u^+(\mbf{x'})-u^+(\mbf{x})}^2-\abs{u^-(\mbf{x'})-u^-(\mbf{x})}^2.\]
Reporting this in \labelcref{lower_bound_energy_g}, we get the result by replacing $\eta$ with $\norm{Q}^{-1}\eta$ and restricting $u$ to $U \coloneqq u^{-1}(B_\rho^{d^+}(0) \times B_\rho^{d^-}(0) \times B_\rho^{d^0}(0))$ for some small $\rho >0$.
\end{proof}

We will use the following positive signature condition:
\begin{equation}\tag{PS$(\kappa)$}\label{signature-condition}
\begin{gathered}
    \text{for every $\mbf{x} \in \mbf X$,}\quad d^+_c(\mbf x) \geq \kappa \quad\text{where}\quad d^{\pm}_c(\mbf x) \coloneqq \max \left\{ d^\pm(g)(\mbf{x})\;|\;g \in G_c\right\}.
    \end{gathered}
    \end{equation}

\begin{proposition}\label{prop:lower-bound}
Let $c\in\contspace^{2}(\mbf X)$ and assume that for every $i \in \{1,\ldots, m\}$, $X_i \subseteq \R^N$ is a $\contspace^2$ sub-manifold of dimension $d_i$ and $\mu_i \in L^\infty(\hdm^{d_i}_{X_i})$ is a probability measure compactly supported in $X_i$. If \labelcref{signature-condition} is satisfied, then there exists a constant $C_*\in [0,\infty)$ such that for every $\eps >0$,
\begin{equation}\label{lower_bound}
\boxed{\MOT_\eps \geq \MOT_0 + \frac \kappa 2 \eps\log(1/\eps) - C_* \eps.}
\end{equation}
\end{proposition}
\begin{proof}
The measures $\mu_i$ being supported on some compact subsets $K_i\subseteq X_i$, consider a family $(\phi_i)_{1\leq i \leq m} \in \prod_{i=1}^m \contspace(K_i)$ of $c$-conjugate Kantorovich potentials. Taking $(\phi_i)_{1\leq i \leq m}$ as competitor in \labelcref{pb:dual_meot_lse}, we get the lower bound:
\begin{align*}
\MOT_\eps &\geq \sum_{i=1}^m \int_{K_i} \phi_i \dd\mu_i -\eps \log\left(\int_{\mbf K} e^{-\frac{E}\eps} \dd \otimes_{i=1}^m \mu_i\right)\\
&= \MOT_0-\eps \log\left(\int_{\mbf K} e^{-\frac{E}\eps} \dd \otimes_{i=1}^m \mu_i\right),
\end{align*}
where $E \coloneqq c-\oplus_{i=1}^m \phi_i$ on $\mbf{K} = \prod_{i=1}^m K_i$ as in \Cref{gap_lemma}. We are going to show that for some constant $C>0$ and for every $\eps >0$,
\[\int_{\mbf K} e^{-E/\eps} \dd \otimes_{i=1}^m \mu_i \leq C \eps^{\kappa/2},\]
which yields \labelcref{lower_bound} with $C_* = \log(C)$.

For every $\mbf{\bar x} \in \mbf K$, we consider a quadratic form $g_{\mbf{\bar x}} \in \{g(\mbf{\bar x}) \;|\; g\in G_c\}$ of signature $(\kappa, d^-, d^0)$, which is possible thanks to \labelcref{signature-condition}, and take a local chart\footnote{Although $U$, $R$, $d^-$ and $d^0$ depend on $\mbf{\bar x}$, we do not index them with $\mbf{\bar x}$ so as to ease notations.}
\[u_{\mbf \bar x} : U \subseteq \mbf X \to B^\kappa_R(0) \times B^{d^-}_R(0)\times B^{d^+}_R(0)\]
as given by \Cref{gap_lemma}, such that \labelcref{gap_inequality} holds with $\eta(r) \leq 1/2$ for every $r$ such that $B_r(\mbf{\bar x}) \subseteq U$. Notice that $u_{\mbf{\bar x}}$ is bi-Lipschitz with some constant $L_{\mbf{\bar x}}$ on $V_{\mbf{\bar x}} \coloneqq u_{\mbf{\bar x}}^{-1}(B^\kappa_{R/2}(0) \times B^{d^-}_{R/2}(0) \times B^{d^0}_{R/2}(0))$.

For every $i\in \{1,\ldots,m\}$ we may write $\mu_i = f_i \hdm^{d_i}_{X_i}$ for some density $f_i : X_i \to \R_+$. By applying several times the co-area formula \cite[Theorem~3.2.22]{federerGeometricMeasureTheory1996} to the projection maps onto $X_i$, we may justify that
\[\hdm^{d}_{\mbf X} = \otimes_{i=1}^m \hdm^{d_i}_{X_i}\quad\text{where}\quad d \coloneqq \sum_i d_i.\]
We set $E_{\mbf{\bar x}} \coloneqq E \circ u_{\mbf{\bar x}}^{-1} : B^\kappa_{R}(0) \times B^{d^-}_{R}(0) \times B^{d^0}_{R}(0) \to [0,+\infty]$ and we apply the area formula:
\begin{align*}
\MoveEqLeft \int_{V_{\mbf{\bar x}}} e^{-E/\eps} \dd \otimes_{i=1}^m \mu_i = \int_{V_{\mbf{\bar x}}} e^{-E/\eps} \otimes_{i=1}^m f_i \dd \hdm^d_{\mbf X} \\
&= \int_{B^\kappa_{R/2}(0) \times B^{d^-}_{R/2}(0) \times B^{d^0}_{R/2}(0)} e^{-E_{\mbf{\bar x}}/\eps} \otimes_{i=1}^m f_i J u_{\mbf{\bar x}}^{-1}\dd \hdm^\kappa \otimes\hdm^{d^-} \otimes\hdm^{d^0}\\
&\leq L_{\mbf{\bar x}} \prod_{i=1}^m \norm{\mu_i}_{L^\infty(\hdm^{d_i}_{X_i})}  \int_{B^\kappa_{R/2}(0) \times B^{d^-}_{R/2}(0) \times B^{d^0}_{R/2}(0)} e^{-E_{\mbf{\bar x}}(u^+,u^-,u^0)/\eps} \dd (u^+, u^-,u^0).
\end{align*}
Now, for every $(u^-,u^0) \in B^{d^-}_{R/2}(0) \times B^{d^0}_{R/2}(0)$, consider a minimizer of $E_{\mbf{\bar x}}(\cdot,u^-,u^0)$ over $\bar B^\kappa_{R/2}(0)$ denoted by $f^+(u^-,u^0)$. By \labelcref{gap_inequality} of \Cref{gap_lemma}, for every $(u^+,u^-,u^0)\in B^\kappa_{R/2}(0) \times B^{d^-}_{R/2}(0) \times B^{d^0}_{R/2}(0)$,
\begin{align*}
E_{\mbf{\bar x}}(u^+,u^-,u^0) &\geq \frac{1}{2} (E_{\mbf{\bar x}}(f^+(u^-,u^0),u^-,u^0)+ E_{\mbf{\bar x}}(u^+,u^-,u^0))\\
&\geq (1- 1/2) \abs{u^+-f^+(u^-,u^0)}^2  =\frac{1}{2} \abs{u^+-f^+(u^-,u^0)}^2.
\end{align*}
As a consequence we obtain:
\begin{align*}
&\qquad \int_{B^\kappa_{R/2}(0) \times B^{d^-}_{R/2}(0) \times B^{d^0}_{R/2}(0)} e^{-E_{\mbf{\bar x}}(u^+,u^-,u^0)/\eps} \dd(u^+,u^-,u^0)\\
&\leq  \int_{B^{d^-}_{R/2}(0) \times B^{d^0}_{R/2}(0)} \int_{B^\kappa_{R/2}(0)} e^{-\frac{\abs{u^+-f^+(u^-,u^0)}^2}{2\eps}} \dd u^+ \dd(u^-,u^0)\\
&\leq \eps^{\kappa/2} \omega_{d^-}\omega_{d^0} R^{d^- +d^0} \int_{\R^\kappa} e^{-\abs{u}^2/2} \dd u = C_{\mbf{\bar x}} \eps^{\kappa/2}
\end{align*}
for some constant $C_{\mbf{\bar x}} > 0$ (which depends on $\mbf{\bar x}$ through $R$, $d^-$ and $d^0$). The sets $\{V_{\mbf{\bar x}}\}_{\mbf{\bar x}\in \mbf\Sigma}$ form an open covering of the compact set $\mbf\Sigma \coloneqq \{\mbf x \in \mbf K\;|\;E(\mbf x) = 0\}$, hence we may extract a finite covering $V_{\mbf{\bar x_1}}, \ldots, V_{\mbf{\bar x_L}}$ and for every $\eps > 0$:
\[\int_{\bigcup_{\ell=1}^L V_{\mbf{\bar x_\ell}}} e^{-E/\eps} \dd \otimes_{i=1}^m \mu_i \leq \eps^{\kappa/2}\Biggl(\sum_{\ell=1}^L L_{\mbf{\bar x_\ell}} C_{\mbf{\bar x_\ell}} \Biggr) \Biggl(\prod_{i=1}^m \norm{\mu_i}_{L^\infty(\hdm^{d_i}_{X_i})}\Biggr) = C_1 \eps^{\kappa/2},\]
for some constant $C_1 \in (0,+\infty)$. Finally, since $E$ is continuous and does not vanish on the compact set $K' \coloneqq \mbf K \setminus \bigcup_{\ell=1}^L V_{\mbf{\bar x_\ell}}$, it is bounded from below on $\mbf{K'}$ by some constant $C_2 > 0$. Therefore, for every $\eps > 0$,
\[\int_{\mbf K} e^{-E/\eps} \dd \otimes_{1\leq i\leq m} \mu_i \leq C_1\eps^{\kappa/2} + e^{-C_2/\eps} \leq  C \eps^{\kappa/2},\]
for some constant $C > 0$. This concludes the proof.
\end{proof}

\section{Examples and matching bound}
\label{sec:examples}
We devote this section to applying the results we have stated above to several cost functions. For simplicity we can assume that  the dimensions of the  $X_i$ are all equal to some common $d$ and the cost function $c$ is $\contspace^2$. As in \cite{passLocalStructureOptimal2012} we consider, for the lower bound, the metric $\overline g$ such that $t_p=\frac{1}{2^{m-1}-1}$ for all $p\in P$, we remind that $P$ is the set of partition of $\{1,\ldots,m\}$ into two non empty disjoint subsets.

\begin{example}[{\bfseries Two marginals case}]
  In previous works \cite{carlierConvergenceRateGeneral2023,ecksteinConvergenceRatesRegularized2023} concerning the rate of convergence for the two marginals problem, it was assumed that the cost function must satisfy a non degeneracy condition, that is $D_{x_1x_2}^2c$ must be of full rank. A direct consequence of our analysis is that we can provide a lower bound (the upper bound does not depend on such a condition) for costs for which the non-degeneracy condition fails. Let $r$ be the rank of $D_{x_1x_2}^2 c$ at the point where the non-degeneracy condition fails, then the signature of $\overline g$ at this point is given by $(r,r,2d-2r)$ meaning that  locally the support of the optimal $\gamma_0$ is at most $2d-r$ dimensional. Thus, the bounds become
  \begin{equation*}
\boxed{\frac r 2 \eps\log(1/\eps) - C_* \eps \leq \OT_\eps - \OT_0 \leq \frac d 2 \eps\log(1/\eps) + C^*\eps,}    
\end{equation*}
for some constants $C_*,C^*>0$. Notice that if $D_{x_1,x_2}^2c$ has full rank then $r=d$ and we retrieve the matching bound results of \cite{carlierConvergenceRateGeneral2023,ecksteinConvergenceRatesRegularized2023}.
\end{example}
\begin{example}[{\bfseries Two marginals case and unequal dimension}]
\label{ex_unequal}
Consider now the two marginals case but unequal dimensional, that is for example $d_1>d_2$. Then, if $D_{x_1,x_2}^2c$ has full rank, that is $r=d_2$, we obtain a matching bound depending only on the lower dimensional marginal
  \begin{equation*}
\boxed{\frac{d_{2}} {2} \eps\log(1/\eps) - C_* \eps \leq \OT_\eps - \OT_0 \leq \frac{d_{2}} {2} \eps\log(1/\eps) + C^*\eps,}    
\end{equation*}
for some constants $C_*,C^*>0$. If $\mu_1$ is absolutely continuous with respect to $\hdm^{d_1}$ on some smooth sub-manifold of dimension $d_1$, then any optimal transport plan would be concentrated on a set of Hausdorff dimension no less than $d_1$, and thus the upper bound given in \cite[Theorem~3.8]{ecksteinConvergenceRatesRegularized2023} would be $\frac{d_1}2\eps\log(1/\eps) + O(\eps)$, which is strictly worse than our estimate.
\end{example}
\begin{example}[{\bfseries Negative harmonic cost}]
Consider the cost $c(x_1,\ldots,x_m)=h(\sum_{i=1}^mx_i)$
where $h$ is $\contspace^2$ and $D^2h>0$. Assuming that the marginals
have finite second  moments, when $h(x)=|x|^2$ this kind of cost is 
equivalent to the harmonic negative cost that is 
$c(x_1,\ldots,x_m)=-\sum_{i<j}|x_i-x_j|^2$ (here $|\cdot|$
denotes the standard euclidean norm), see 
\cite{dimarinoOptimalTransportationTheory2017} for more
details. 
It follows now that the signature of the metric $\overline g$ is $(d,(m-1)d,0)$ thus the bounds between $\MOT_\eps$ and $\MOT_0$ that we obtain are 
\begin{equation*}
\boxed{\frac d 2 \eps\log(1/\eps) - C_* \eps \leq \MOT_\eps - \MOT_0 \leq \frac 1 2\Bigl((m-1)d\Bigr) \eps\log(1/\eps) + C^*\eps,}    
\end{equation*}
for some constants $C_*,C^*>0$. We remark that it is known from \cite{passLocalStructureOptimal2012,dimarinoOptimalTransportationTheory2017} that a transport plan $\gamma_0$ is optimal if and only if it is supported on the set $\{(x_1,\ldots,x_m)\;|\;\sum_{i=1}^mx_i=l\}$, where $l\in\R^d$ is any constant and there exists solutions  whose support has dimension exactly $(m-1)d$. 
\end{example}

\begin{example}[{\bfseries Gangbo-Święch cost and Wasserstein barycenter}]
Suppose that $c(x_1,\ldots,x_m)=\sum_{i<j}|x_i-x_j|^2$, known as the Gangbo-Święch cost \cite{gangboOptimalMapsMultidimensional1998}. Notice that the cost is equivalent  to $c(x_1,\ldots,x_m)=h(\sum_{i=1}^m x_i)$ where $h$ is $\contspace^2$ and $D^2h <0$,then the signature of $\overline g$ is $((m-1)d,d,0)$ and we have a matching bound
\begin{equation*}
\boxed{\frac 1 2 \Bigl((m-1)d\Bigr)\eps\log(1/\eps) - C_* \eps \leq \MOT_\eps - \MOT_0 \leq \frac 1 2\Bigl((m-1)d\Bigr) \eps\log(1/\eps) + C^*\eps.}    
\end{equation*}
Notice now that considering the $\MOT_0$ problem with a cost
$c(x_1,\ldots,x_m)=\sum_{i}|x_i-T(x_1,\ldots,x_m)|^2$, where
$T(x_1,\ldots,x_m)=\sum_{i=1}^m\lambda_ix_i$ is the Euclidean barycenter , is equivalent to the $\MOT_0$ with the Gangbo-
Święch cost and the matching bound above still holds. 
Moreover, the multi-marginal problem with this particular cost
has been shown \cite{aguehBarycentersWassersteinSpace2011} to be 
equivalent to the Wasserstein barycenter, that is $T_\sharp\gamma_0=\nu$
is the barycenter of $\mu_1,\ldots,\mu_m$.
\end{example}

\smallskip

\noindent{\textbf{Acknowledgments.}} L.N. is partially on academic leave at Inria (team Matherials) for the year 2022-2023 and acknowledges the hospitality if this institution during this period. His work was supported by a public grant as part of the Investissement d'avenir project, reference ANR-11-LABX-0056-LMH, LabEx LMH and  from H-Code, Université Paris-Saclay. P.P. acknowledges the academic leave provided by Inria Paris (team MOKAPLAN) for the year 2022-2023. Both authors acknowledge the financial support by the ANR project GOTA (ANR-23-CE46-0001).

\printbibliography
\end{document}